\documentclass[11pt, reqno]{amsart}
  \usepackage{amsmath, amsfonts, amssymb,amsthm}
\numberwithin{equation}{section}
\newtheorem{theorem}{Theorem}
\newtheorem{lemma}{Lemma}
\newtheorem{cor}{Corollary}

\theoremstyle{definition}

\newcommand{\C}{\mathbb{C}}

\newcommand{\N}{\mathbb{N}}

\begin{document}
\title{ A note on  Hayman's conjecture}
\author{Ta Thi Hoai An}
\address{Institute of Mathematics, Vietnam Academy of Science and Technology\\
18 Hoang Quoc Viet Road, Cau Giay District \\
10307 Hanoi,  Vietnam}
\address{and: Institute of Mathematics and Applied Sciences (TIMAS)\\ Thang Long University\\ Hanoi,  Vietnam}
\email{tthan@math.ac.vn}
\author{Nguyen Viet Phuong}
\address{Thai Nguyen University of Economics and Business Administration, Vietnam} \email{nvphuongt@gmail.com}

\thanks{Key words: Meromorphic Functions, Entire Functions, Nevanlinna theory, Value Distribution,  Differential Polynomial.}

\begin{abstract}
In this paper, we will give suitable conditions  on differential polynomials $Q(f)$   such that they take every finite non-zero value infinitely often, where $f$ is a meromorphic function in complex plane. These results are related to Problem 1.19 and Problem 1.20 in a book of Hayman and Lingham \cite{HL}. As consequences, we give a new proof of the Hayman conjecture.  Moreover, our results allow  differential polynomials $Q(f)$ to have some terms of any degree of $f$
and also the hypothesis $n>k$ in  \cite[Theorem 2]{BE} is replaced by    $n\ge 2$ in our result.
 \end{abstract}

\thanks{The  authors are supported  by Vietnam's National Foundation for Science and Technology Development (NAFOSTED), under grant number  101.04-2017.320}

\baselineskip=15.5truept 
\maketitle 
\pagestyle{myheadings}
\markboth{}{}

\section{ Introduction and main results }

In 1940, Milloux obtained the first result involving the value distribution of meromorphic functions together with their derivatives and called usually Milloux's inequality. This result has led to a series of important research questions, for example Hayman \cite{H1} (or Problem 1.19 in \cite{HL}) proved that if $f$ is a meromorphic function in the plane omitting a finite value $a$, and if its $k$-th derivative $f^{(k)}$, for some $k\geq 1$, omits a finite nonzero value $b$, then $f$ is constant. This result is known as {\it Hayman's alternative}. 
Closely linked to Milloux's inequality and Hayman's alternative is Hayman's conjecture on the value distribution of certain differential monomials,  if $f$ is
a transcendental meromorphic function and $n\in \N$, then Hayman conjectured that $(f^{n})'$ takes every finite nonzero value infinitely often for all $n\ge 2$. Hayman \cite{H1} himself proved this conjecture for $n \ge 4,$  Mues \cite{M} proved the case of $n = 3$. 
In 1995, Bergweiler and Eremenko \cite{BE},  Chen and Fang  \cite{CF}, and Zalcman  \cite{Za} independently proved the conjecture for $n = 2.$ 
In \cite{BL, Cl, He, M, Za}, the authors generalized Hayman's conjecture to $f^{(k)}$. The effective result in this direction was given in \cite[Theorem 2]{BE}: if $f$ is a transcendental
meromorphic function in the plane and $m > k \ge 1 $ then $(f^m)^{(k)}$ assumes every finite non-zero value
infinitely often. A closed connection was given in   \cite[Problem 1.20]{HL}, which was proven for $n\ge 5$ in \cite{H1}, $n=4$ in \cite{M} and $n=3$ in \cite{BL}: If $f$ is non-constant and meromorphic in the plane and $n\ge 3$, then $f'-f^n$ assumes all finite complex values.

  A question arising in connection with Hayman's alternative was given by Eremenko and Langley \cite{EL}: whether $(f^m)^{(k)}$ can be replaced by a more general term, such as a linear differential polynomial $$F=L[f]=f^{(k)}+a_{k-1}f^{(k-1)}+\dots +a_0f,$$ where  the coefficients $a_j$ are small functions of $f.$ The following counterexample gives a negative answer to above question. Let $f(z)=1+e^{-z^2}$ and $F=L[f]=f''+2zf'+3f=3+e^{-z^2}.$ Then $f(z)\ne 1$ and $F(z)\ne 3$, but $f$ is not constant.
 Until now, most  results related to Hayman's conjecture and Hayman's inequality  were considered for   non-constant differential polynomials $Q$ in $f$, where all of
whose terms have degree at least 2 in $f$ and its derivatives (see \cite{C}, \cite{Hi}).

 In this paper, we will  generalize Hayman's conjecture to differential polynomials. Here, terms of the polynomials are not necessary all to have degree at least 2.
 
Let $P(z)$ and $Q(z)$ be polynomials in $\C[z]$ of degree $p$ and $q$   respectively. 
 We write
\begin{align*}&Q(z) = b(z-\alpha_1)^{m_1}(z-\alpha_2)^{m_2}\dots(z-\alpha_l)^{m_l},\quad \text{and}\\
&P(z) = c(z-\beta_1)^{n_1}(z-\beta_2)^{n_2}\dots(z-\beta_h)^{n_h}.\end{align*} where $b,c\in\C^*.$

Our results are as follows.

\begin{theorem}\label{th1}
Let $k$ be an positive integer.  If $f$ is a transcendental meromorphic function and $q\ge l+1,$ then $[Q(f)]^{(k)}$ takes every finite non-zero value infinitely often.
\end{theorem}

%__________________Một số hệ quả_________________

 As a consequence,  when we consider $k=1$ and   $Q(z)=z^{n}$,  Hayman conjecture is obtained:

\begin{cor}[Hayman conjecture \cite{BE,H1,H,M} ] \label{corollary1} 
 If $f$ is a transcendental meromorphic function, then $f^nf'$ takes every finite non-zero complex value infinitely often, for any integer $n\ge 1$.
 \end{cor}
 
 More generally, for any $n\ge 2$:
 \begin{cor}\cite[Theorem 2]{BE}\label{corollary2} If $f$ is a transcendental meromorphic function, and if $n\ge 2$ and $k$ are positive integers, then $(f^n)^{(k)}$ takes every finite non-zero complex value infinitely often.
\end{cor}
 
 Note that in  \cite[Theorem 2]{BE}, the authors need $n>k$. However, in Corollary~\ref{corollary2}, we only need    $n\ge 2$.
 
 \begin{cor} Let $k, l$ be positive integers and $a_1,\dots, a_l$ be complex numbers.
If $f$ is a transcendental meromorphic function  then $[(f-a_1)^{2}(f-a_2)\dots(f-a_l)]^{(k)}$ takes every finite non-zero value infinitely often.
\end{cor}
%_________________________Định lý 3_____________
\begin{theorem}\label{th3} Let $P$ and $Q$ be polynomials of degree $p$ and $q$ respectively, and $k\geq 1$ be a positive integer. Let $f$ be a transcendental meromorphic function.
%  Assume that 
% $$Q(z) = b(z-\alpha_1)^{m_1}(z-\alpha_2)^{m_2}\dots(z-\alpha_l)^{m_l}.$$ 
  If $q\geq (k+1)p+l+2$ then $Q(f)+P(f^{(k)})$ has infinitely many zeros.

\end{theorem}

 In the special case that $P(z)=z$  and  $Q(z)=-az^{n}+b$ where $a\ne 0,b$ are constants, we recover many known results, for
example the results in \cite{H1}, as special cases of our result.
\begin{cor}\cite[Theorem 9]{H1}\label{corollary4} If $f$ is a transcendental meromorphic function, $n\ge 5$ and $a\ne 0$, then $f'(z)-af(z)^n$ takes every finite complex value infinitely often.
\end{cor}

\begin{theorem}\label{th2} Let $P$ and $Q$ be polynomials of degree $p$ and $q$ respectively. Let $\alpha\ne 0$ be a small function with respect to $f$.
 If $f$ is a transcendental meromorphic function of finite order and $p-q-h-2l-2>0$, then $P(f)Q(f(z+c_1)+c_2f(z))-\alpha$ has infinitely many zeros, for any $c_1,c_2\in \C$.

In particular, if $f$ is a transcendental entire function of finite order and $p-h-l>0$, then $P(f)Q(f(z+c_1)+c_2f(z))-\alpha$ has infinitely many zeros, for any $c_1,c_2\in \C$.
\end{theorem}

%\begin{remark} The condition $``f$ is meromorphic function of finite order" in Theorem~\ref{th2} can not be removed. 
%\end{remark}
 \begin{cor}\cite[Theorem 1.2]{LCL}\label{corollary3} Suppose that $f$ is a transcendental meromorphic function, $\alpha$ is a small function with respect to $f$ and  $n,s$ are integers. If $n\ge s+6$, then $f(z)^n(f(z+c)-f(z))^s-\alpha$ takes every finite non-zero complex value infinitely often.
\end{cor}

 {\it Acknowledgments.} A part of this article was written while the first name
author was visiting  Vietnam Institute for Advanced Study in Mathematics (VIASM).  She would like to thank the institute for warm hospitality and partial support. The first name author also would like to thank the International Centre for Research and Postgraduate Training in Mathematics (ICRTM, grant number ICRTM01-2020.05) for a part of support.

\section{Proof of theorems}
Let  $f$ be a meromorphic function on  the complex plane $\C.$ We use standard notations, definitions and results of Nevanlinna theory in \cite{H, R}. We denote by $S(r,f)$ any function satisfying $S(r,f)=o(T(r,f))$ as $r\rightarrow +\infty$ outside of a possible exceptional set with finite measure.  
 We first recall the following lemmas

\begin{lemma}\cite[Theorem 2.1]{CF} \label{lemma3.5}   Let $f(z)$ be a transcendental meromorphic function of finite order. Then, for any $c\in \C$,
\begin{itemize}\item[(i)]  $m\Big(r,\frac{f(z+c)}{f(z)}\Big )=S(r,f)$ for all $r$ outside of a set of finite logarithmic measure.
\item[(ii)]T(r,f(z+c))=T(r,f(z))+S(r,f).
\end{itemize}
\end{lemma}

\begin{lemma} \label{lemma3.2} (Milloux, see Hayman \cite[Theorem 3.2]{H}) For $k\ge 1,$
\begin{equation*}T(r,f)\leq \overline{N}(r,f)+N\Big(r,\dfrac{1}{f}\Big)+N\Big(r,\dfrac{1}{f^{(k)}-1}\Big)-N\Big(r,\dfrac{1}{f^{(k+1)}}\Big)+S(r,f).
\end{equation*}
\end{lemma}

Recently, Yamanoi \cite{Y1} proved the following result.

\begin{lemma}\label{lemma3.3}\cite{Y1} Let $f$ be a transcendental meromorphic function in the complex plane, $k\ge 1$ be an integer, and $\epsilon>0;$ let $A\subset \C$ be a finite set of complex numbers. Then we have 
\begin{equation*}
k\overline{N}(r,f)+\sum_{a\in A} N_1\Big(r,\dfrac{1}{f-a}\Big)\leq N\Big(r,\dfrac{1}{f^{(k+1)}}\Big )+\epsilon T(r,f),
\end{equation*}
for all $r>e$ outside a set $E\subset (e,\infty )$ of logarithmic density $0.$ Here, $E$ depends on $f,k,\epsilon$ and $A,$ and $$N_1\Big(r,\dfrac{1}{f-a}\Big)=N\Big(r,\dfrac{1}{f-a}\Big)-\overline{N}\Big(r,\dfrac{1}{f-a}\Big).$$
\end{lemma}

\begin{proof}[{Proof of Theorem~\ref{th1}}] We will prove  $[Q(f)]^{(k)}$ takes a non-zero constant $a$ infinitely often. Without loss of generality, we may assume $a=1$.
Applying Lemma~\ref{lemma3.3} to the transcendental meromorphic function $Q(f)$ and $k\ge 1,$ set $A=\{0\}\subset\C$ and $\epsilon=\dfrac{1}{2q},$ we have
 \begin{align*}\label{ct1}
k\overline{N}(r,f)+N\Big(r,\dfrac{1}{Q(f)}\Big)-\overline{N}\Big(r,\dfrac{1}{Q(f)}\Big)&\leq N\Big(r,\dfrac{1}{Q(f)^{(k+1)}}\Big )+\dfrac{1}{2q} T(r,Q(f))\notag\\
&\leq N\Big(r,\dfrac{1}{Q(f)^{(k+1)}}\Big )+\dfrac{1}{2} T(r,f)+O(1)
\end{align*}

Together with Lemma~\ref{lemma3.2}, we get
\begin{align*}
T(r,Q(f))&\le k\overline{N}(r,f) +N\Big(r,\dfrac{1}{Q(f)}\Big)-\overline{N}\Big(r,\dfrac{1}{Q(f)}\Big)\\
&\quad-(k-1)\overline{N}(r,f)+\overline{N}\Big(r,\dfrac{1}{Q(f)}\Big)+N\Big(r,\dfrac{1}{(Q(f))^{(k)}-1}\Big)\\
&\quad-N\Big(r,\dfrac{1}{(Q(f))^{(k+1)}}\Big)+S(r,f)\\
&\le \overline{N}\Big(r,\dfrac{1}{Q(f)}\Big)+N\Big(r,\dfrac{1}{(Q(f))^{(k)}-1}\Big)+\dfrac{1}{2} T(r,f)+S(r,f)\\
&\le\sum_{i=1}^{l}\overline{N}\Big(r,\dfrac{1}{f-\alpha_i}\Big)+N\Big(r,\dfrac{1}{(Q(f))^{(k)}-1}\Big)+\dfrac{1}{2} T(r,f)+S(r,f)\\
&\le \Big(l+\dfrac{1}{2}\Big)T(r,f)+N\Big(r,\dfrac{1}{(Q(f))^{(k)}-1}\Big)+S(r,f).
\end{align*}

Hence,
\begin{equation*}\Big(q-l-\dfrac{1}{2}\Big)T(r,f)\le N\Big(r,\dfrac{1}{(Q(f))^{(k)}-1}\Big)+ S(r,f).
\end{equation*}

Thus $(Q(f))^{(k)}=1$ has infinitely many roots when $q\ge l+1.$ It follows that $(Q(f))^{(k)}$ assumes every finite non-zero value infinitely often. The proof is complete.
\end{proof}

\begin{proof}[{Proof of Theorem~\ref{th3}}]

Set 
\begin{align}&F=Q(f)+P(f^{(k)}),\label{ct1}\\
&R(f)=\frac{[Q(f)]'}{Q(f)}-\frac{F'}{F},\quad H(f)=P(f^{(k)})\Big(\frac{F'}{F}-\frac{[P(f^{(k)})]'}{P(f^{(k)})}\Big ).
\end{align}
 We have 
 \begin{align}m(r,R(f))=S(r,f),\label{m2}\\
Q(f)R(f)=H(f),\label{ct2}
\end{align} 
and
\begin{align*}T(r,F)&\leq T(r,Q(f))+T(r,P(f^{(k)})+S(r,f)=qT(r,f)+pT(r,f^{(k)})+S(r,f)\\
&\leq (q+p(k+1)) T(r,f)+S(r,f).\end{align*}
Hence $S(r,F)=S(r,f)$.

If $R(f)\equiv 0,$ then $H(f)\equiv 0$. This means
 \begin{equation} Q(f)=(c-1)P(f^{(k)}), \label{ct0}\end{equation} where $c\ne 0, 1$ is a constant. Since $q\geq (k+1)p+l+2$, Equation  \eqref{ct0} implies that $f$ cannot have poles. On the other hand, we have 
\begin{align*}qm(r,f)=m(r,Q(f))&\leq m(r,P(f^{(k)})+O(1)\leq pm(r,f^{(k)})+O(1)\\
&\leq pm(r,f)+pm\big(r,\frac{f^{(k)}}{f}\big)+O(1)=pm(r,f)+S(r,f).
\end{align*}
Therefore \begin{equation}m(r,f)=S(r,f),\label{ct01}\end{equation}
and
 $$T(r,f)=N(r,f)+m(r,f)=S(r,f),$$ which is a contradiction and then $R(f)\not\equiv 0.$
 
It is easy  to see that 
 \begin{align}m(r,H(f))&\leq m(r,P(f^{(k)}))+m\Big(r,\frac{F'}{F}-\frac{[P(f^{(k)})]'}{P(f^{(k)})}\Big )+O(1)\notag\\
&\leq pm(r,f^{(k)})+S(r,f)\leq p\Big (m(r,f)+m(r,\frac{f^{(k)}}{f})\Big)+S(r,f)\notag\\
&=pm(r,f)+S(r,f).\label{ct3}
\end{align}
We have
\begin{align*}qm(r,f)&=m(r,Q(f))+S(r,f)=m\Big(r,Q(f)R(f)\frac{1}{R(f)}\Big)+S(r,f)\\
&\leq m(r,Q(f)R(f))+m\Big(r,\frac{1}{R(f)}\Big)+S(r,f)\\
&=m(r,H(f))+m\Big(r,\frac{1}{R(f)}\Big)+S(r,f)\\
&\leq pm(r,f)+m\Big(r,\frac{1}{R(f)}\Big)+S(r,f),
\end{align*}
hence
\begin{equation}\label{ct4} (q-p)m(r,f)\leq m\Big(r,\frac{1}{R(f)}\Big)+S(r,f).\end{equation}
By the first main theorem and \eqref{m2}, we have
\begin{equation}\label{ct5}m(r,\frac{1}{R(f)})=T(r,R(f))-N(r,\frac{1}{R(f)})+O(1)=N(r,R(f))-N(r,\frac{1}{R(f)})+S(r,f).\end{equation}

From the definition of $R(f)$, we see immediately that the possible poles of $R(f)$ occur only at the poles of $f$ and the zeros of $Q(f)$ and $F.$ Note that $R(f)$ can have only simple poles. Now, suppose that $z_0$ is a pole of $f$ of order $s.$ Then $z_0$ is a pole of $Q(f)$ of order $qs$ and $H(f)$ of order at most $(s+k)p+1.$ Since $q\geq (k+1)p+l+2,$ we have $$qs-(s+k)p-1=(q-p)s-kp-1>0 .$$ Thus, by \eqref{ct2}, we deduce $z_0$ must be a zero of $R(f)$ of order at least $$qs-((s+k)p+1)=(q-p)s-kp-1.$$ 
Hence, we get 
\begin{align}N(r,R(f))&\leq \overline{N}\big(r,\frac{1}{F}\big)+\overline{N}\big(r,\frac{1}{Q(f)}\big)\leq \overline{N}\big(r,\frac{1}{F}\big)+\sum_{i=1}^{l}\overline{N}\big(r,\frac{1}{f-\alpha_i}\big)\cr
&\leq \overline{N}\big(r,\frac{1}{F}\big)+lT(r,f)+O(1),\label{t1}
\end{align}
and \begin{align}\label{ct6}N\big(r,\frac{1}{R(f)}\big)\geq (q-p)N(r,f)-(kp+1)\overline{N}(r,f).
\end{align}
Combining \eqref{ct4}, \eqref{ct5}, \eqref{t1}, \eqref{ct6} and by the first main theorem, we get
\begin{align*}(q-p)T(r,f)&=(q-p)N(r,f)+(q-p)m(r,f)\\
&\leq \overline{N}\big(r,\frac{1}{F}\big)+lT(r,f)+(kp+1)\overline{N}(r,f)+S(r,f)\\
&\leq \overline{N}\big(r,\frac{1}{F}\big)+(l+kp+1)T(r,f)+S(r,f).
\end{align*}
Hence, $$\big(q-(k+1)p-l-1\big)T(r,f)\leq \overline{N}\big(r,\frac{1}{F}\big)+S(r,f).$$

Thus $F=Q(f)+P(f^{(k)})$ assumes every finite value infinitely often when $q\geq (k+1)p+l+2.$

\end{proof}

\begin{proof}[{Proof of Theorem~\ref{th2}}] Denote by $G(z)=P(f(z))Q(f(z+c_1)+c_2f(z)).$ Applying the Second Main Theorem for the meromorphic function $G$ and  $0,\infty$, $\alpha$, we have
\begin{align}T(r,G)&\le \overline{N}(r,G)+\overline{N}(r,\frac 1G)+\overline{N}(r,\frac1{G-\alpha}) +S(r,f)\cr
&\le 2{T}(r,f) +\sum_{i=1}^h\overline{N}(r,\frac 1{f-\beta_i}) +\sum_{i=1}^l\overline{N}(r,\frac 1{f(z+c_1)+c_2f(z))-\alpha_i})\cr
&\quad+\overline{N}(r,\frac1{G-\alpha})+S(r,f)\cr
&\le (h+2l+2){T}(r,f) +\overline{N}(r,\frac1{G-\alpha}) +S(r,f).\label{4}
\end{align}

On the other hand, we have
\begin{align}\frac 1{f^qP(f)}&=\frac1{G}\frac{Q(f(z+c_1)+c_2f(z))}{f^q}\notag\\
&=\frac{1}{G}\prod_{i=1}^{l}\Big[\frac{f(z+c_1)+c_2f(z)-\alpha_i}{f(z)}\Big]^{m_i}.\label{p1}
\end{align} 
Therefore,
\begin{align*}(p+q)T(r,f)&=T(r,f^qP(f))+O(1)\\
&\le T(G)+\sum_{i=1}^{l}m_i T\Big(r,\frac{f(z+c_1)+c_2f(z)-\alpha_i}{f(z)}\Big)+O(1)\cr
&\le T(G)+\sum_{i=1}^{l}m_i T\Big(r,\frac{f(z+c_1)-\alpha_i}{f(z)}\Big)+O(1)\\
&\le T(G)+2\sum_{i=1}^{l}m_i T(r,f)+O(1)\\
&\le T(G)+2q T(r,f)+O(1),
\end{align*}
which implies 
\begin{align}T(r,G)\ge (p-q)T(r,f)+O(1).\label{5}
\end{align}

Combining (\ref{4}) and (\ref{5}), we have
\begin{align*}(p-q-h-2l-2)T(r,f)\le \overline{N}(r,\frac1{G-\alpha}) +S(r,f).
\end{align*}
Thus, $P(f)Q(f(z+c_1)+c_2f(z))-\alpha$ has infinitely many zeros  if $f$ is a transcendental meromorphic function of finite order and $p-q-h-2l-2>0$.

In particular, if $f$ is an entire function then Lemma~\ref{lemma3.5} (i) implies
\begin{align*}T(r,f(z+c_1)+c_2f(z))-\alpha_i)&=T(r,f(z+c_1)+c_2f(z))+O(1)\cr
&=m(r,f(z+c_1)+c_2f(z))+O(1)\cr
&=m\Big(r,f(z)\big(\frac{f(z+c_1)}{f(z)}+c_2\big)\Big)+O(1)\cr
&\leq m(r,f)+m\Big(r,\frac{f(z+c_1)}{f(z)}+c_2\Big)+O(1)\cr
&\leq m(r,f)+m\Big(r,\frac{f(z+c_1)}{f(z)}\Big)+O(1)\cr
&\leq T(r,f)+S(r,f).
\end{align*}
Together with (\ref{4}), we have 
\begin{align}T(r,G)&\le \overline{N}(r,\frac 1G)+\overline{N}(r,\frac1{G-\alpha}) +S(r,f)\cr
&\le \sum_{i=1}^h\overline{N}(r,\frac 1{f-\beta_i}) +\sum_{i=1}^l\overline{N}(r,\frac 1{f(z+c_1)+c_2f(z)-\alpha_i})\cr
&\quad+\overline{N}(r,\frac1{G-\alpha})+S(r,f)\cr
&\le \sum_{i=1}^h\overline{N}(r,\frac 1{f-\beta_i}) +\sum_{i=1}^lT(r,f(z+c_1)+c_2f(z)-\alpha_i)\cr
&\quad+\overline{N}(r,\frac1{G-\alpha})+S(r,f)\cr
&\le (h+l){T}(r,f) +\overline{N}(r,\frac1{G-\alpha}) +S(r,f).\label{6}
\end{align}

On the other hand, by \eqref{p1}, we obtain
\begin{align*}(p+q)T(r,f)&=T(r,f^qP(f))+O(1)\\
&\le T(G)+\sum_{i=1}^{l}m_i T\Big(r,\frac{f(z+c_1)+c_2f(z)-\alpha_i}{f(z)}\Big)+O(1)\cr
&\le T(G)+\sum_{i=1}^{l}m_i T\Big(r,\frac{f(z+c_1)-\alpha_i}{f(z)}\Big)+O(1)\\
&\le T(G)+\sum_{i=1}^{l}m_i T(r,f)+O(1)\\
&\le T(G)+q T(r,f)+O(1),
\end{align*}
where the fourth inequality follows from 
\begin{align*}T\Big(r,\frac{f(z+c_1)-\alpha_i}{f(z)}\Big)&=m\Big(r,\frac{f(z+c_1)-\alpha_i}{f(z)}\Big)+N\Big(r,\frac{f(z+c_1)-\alpha_i}{f(z)}\Big)\\
&\leq m\Big(r,\frac{f(z+c_1)}{f(z)}\Big)+m\Big(r,\frac{1}{f(z)}\Big)+N(r,f(z+c_1)-\alpha_i)\\
&\quad+N\Big(r,\frac{1}{f(z)}\Big)+O(1)\\
&\leq T(r,f)+S(r,f).
\end{align*}
Hence, we get
\begin{align}T(r,G)\ge pT(r,f)+S(r,f).\label{5.1}
\end{align}
Combining \eqref{6} and \eqref{5.1}, we have
$$(p-h-l)T(r,G)\le \bar{N}(r,\frac1{G-\alpha}) +S(r,f).$$

Hence, $P(f)Q(f(z+c_1)+c_2f(z))-\alpha$ has infinitely many zeros when $p-h-l>0$ if $f$ is a transcendental entire function of finite order.

\end{proof}

\end{document}